\documentclass[a4paper]{article}

\usepackage[numbers]{natbib}
\usepackage{natbibspacing}
\renewcommand{\bibfont}{\small}

\usepackage[mlmodern, varphi, scr-cats, sf-cats]{basemath}
\usepackage[vmargin=1in, hmargin=1.6in]{geometry}

\usepackage{bbm}
\usepackage{xcolor}

\usepackage{caption}
\usepackage{subcaption}

\newcommand{\demph}{\textbf}

\theoremstyle{plain}
\newtheorem{theorem}{Theorem}[section]
\newtheorem*{theorem*}{Theorem}
\newtheorem{lemma}[theorem]{Lemma}
\newtheorem{proposition}[theorem]{Proposition}
\newtheorem{corollary}[theorem]{Corollary}

\theoremstyle{definition}
\newtheorem{definition}[theorem]{Definition}
\newtheorem{remark}[theorem]{Remark}
\newtheorem{example}[theorem]{Example}

\tikzset{shorten <>/.style={shorten >=#1,shorten <=#1}}
\tikzset{close/.style={outer sep=-1ex}}

\DeclareFontFamily{U}{mathx}{}
\DeclareFontShape{U}{mathx}{m}{n}{<-> mathx10}{}
\DeclareSymbolFont{mathx}{U}{mathx}{m}{n}
\DeclareMathAccent{\widehat}{0}{mathx}{"70}
\DeclareMathAccent{\widecheck}{0}{mathx}{"71}

\newcommand{\hy}{\mbox{-}}
\newcommand{\coo}{\mathrm{co}}
\renewcommand{\coo}{\mathsf{co}}
\renewcommand{\opp}{\mathsf{op}}

\newcommand{\Prof}{\mathsf{Prof}}
\newcommand{\Mon}{\mathsf{Mon}}
\newcommand{\Grp}{\mathsf{Grp}}

\makeatletter
\def\slashedarrowfill@#1#2#3#4#5{%
  $\m@th\thickmuskip0mu\medmuskip\thickmuskip\thinmuskip\thickmuskip
  \relax#5#1\mkern-7mu%
  \cleaders\hbox{$#5\mkern-2mu#2\mkern-2mu$}\hfill
  \mathclap{#3}\mathclap{#2}%
  \cleaders\hbox{$#5\mkern-2mu#2\mkern-2mu$}\hfill
  \mkern-7mu#4$%
}
\def\rightslashedarrowfill@{%
  \slashedarrowfill@\relbar\relbar\mapstochar\rightarrow}
\newcommand\xslashedrightarrow[2][]{%
  \ext@arrow 0055{\rightslashedarrowfill@}{#1}{#2}}
\makeatother

\newcommand{\slashedrightarrow}{\xslashedrightarrow{}}

\title{Cauchy density}
\author{Adrián Doña Mateo\thanks{School of Mathematics, University of Edinburgh, \href{mailto:adrian.dona@ed.ac.uk}{adrian.dona@ed.ac.uk}}}
\date{}

\begin{document}

\maketitle

\begin{abstract}
    In the paper where he defined the Cauchy completion of a $\cat{V}$-category, Lawvere also defined a condition on a $\cat{V}$-functor which made it analogous to a map of metric spaces whose image is topologically dense in its codomain. We call this condition Cauchy density. In this note, we focus on the fully faithful Cauchy dense $\cat{V}$-functors, and show that the Cauchy completion of $\cat{A}$ is the largest $\cat{V}$-category that admits a fully faithful Cauchy dense $\cat{V}$-functor from $\cat{A}$. Moreover, we show that $F \colon \cat{A} \to \cat{B}$ is fully faithful and Cauchy dense iff $[F,\cat{C}] \colon [\cat{B},\cat{C}] \to [\cat{A},\cat{C}]$ is an equivalence for any Cauchy complete $\cat{C}$. Finally, we provide examples and characterisations of Cauchy dense functors in various contexts.
\end{abstract}

\section{Introduction}\label{sec:intro}

In his 1973 paper~\cite{Lawvere1973}, Lawvere explores the view of metric spaces as categories enriched over the quantale of extended nonnegative real numbers $\R^+ = ([0,\infty], \geq, +, 0)$. One of the remarkable achievements of this perspective is that it allows the generalisation of the notion of the completion of a metric space to the setting of $\cat{V}$-categories, a process that has become known as Cauchy completion. Given a $\cat{V}$-category $\cat{A}$, its Cauchy completion $\overline{\cat{A}}$ is the free cocompletion of $\cat{A}$ under the class of absolute weights, i.e.\ those (small) weights $W \colon \cat{K}^\opp \to \cat{V}$ whose colimits are preserved by any $\cat{V}$-functor.

Before introducing the notion of Cauchy completion, Lawvere defines a condition on a $\cat{V}$-functor which in the case $\cat{V} = \R^+$ corresponds to having a dense image in the topological sense. Explicitly, a functor $F \colon \cat{A} \to \cat{B}$ between two small $\cat{V}$-categories is \demph{Cauchy dense} if for all $b, b' \in \cat{B}$ the morphism
\[\int^a \cat{B}(Fa,b') \otimes \cat{B}(b,Fa) \to \cat{B}(b,b')\]
induced by the composition of $\cat{B}$ is an isomorphism.

Cauchy dense functors have been identified by Lucatelli Nunes and Sousa~\cite[Thm~5.6]{LucatelliNunes2022} as the lax epimorphisms in the 2-category $\cat{V}\hy\Cat$, i.e.\ those 1-cells $F \colon \cat{A} \to \cat{B}$ such that for all $\cat{C} \in \cat{V}\hy\Cat$ the ordinary functor $\cat{V}\hy\Cat(F,\cat{C}) \colon \cat{V}\hy\Cat(\cat{B},\cat{C}) \to \cat{V}\hy\Cat(\cat{A},\cat{C})$ is fully faithful.

In this short note, we turn our attention to the fully faithful Cauchy dense functors, which are the analogue in the theory of $\cat{V}$-categories to the inclusions of the dense subsets of a metric space. Section~\ref{sec:CD} gives the relevant definitions and revisits (and slightly generalises) some of the results in~\cite{LucatelliNunes2022}. Section~\ref{sec:ffCD} deals with fully faithful Cauchy dense functors, and shows that the Cauchy completion of a small $\cat{V}$-category $\cat{A}$ is the largest $\cat{V}$-category which admits a fully faithful Cauchy dense functor from $\cat{A}$ (Theorem~\ref{thm:ffCD}). This is analogous to a result proven by Avery and Leinster~\cite{Avery2021} relating adequate functors and the reflexive completion. We also give an alternative characterisation of fully faithful Cauchy dense functors as those $F \colon \cat{A} \to \cat{B}$ such that $[F,\cat{C}] \colon [\cat{B},\cat{C}] \to [\cat{A},\cat{C}]$ is an equivalence for any Cauchy complete $\cat{C}$ (Theorem~\ref{thm:ffCD_iff_pre_equiv}). Finally, Section~\ref{sec:ex} gives examples of Cauchy dense functors, and characterises them explicitly in a number of different contexts.

As is customary when dealing with enriched category theory, every category, functor and natural transformation will be assumed to be $\cat{V}$-enriched unless otherwise specified, where $\cat{V}$ is a complete and cocomplete symmetric monoidal closed category. The corresponding $\Set$-enriched notions will be termed ordinary. Given a $\cat{V}$-category $\cat{A}$, we write $\cat{A}_0$ for the underlying ordinary category of $\cat{A}$, and similarly for $\cat{V}$-functors. Moreover, in the spirit of Kelly, all Kan extensions considered will be assumed to be pointwise.

\vspace{1em}
\noindent\textbf{Acknowledgements.} I would like to thank Tom Leinster for many interesting discussions, and for providing comments on an initial draft of this note. This work has been funded by a PhD Scholarship from The Carnegie Trust for the Universities of Scotland.

\section{Cauchy dense functors}\label{sec:CD}

The definition of Cauchy density is best motivated using profunctors. We give a brief introduction to these, mostly to fix our notation.

Given two small categories $\cat{A}$ and $\cat{B}$, a \demph{profunctor} $\phi \colon \cat{A} \slashedrightarrow \cat{B}$ is a functor $\cat{B}^\opp \otimes \cat{A} \to \cat{V}$. Every functor $F \colon \cat{A} \to \cat{B}$ induces two profunctors in opposite directions
\[\begin{aligned}
    F_* &\colon \cat{A} \slashedrightarrow \cat{B} \qquad  & F_*(b,a) &= \cat{B}(b,Fa), \\
    F^* &\colon \cat{B} \slashedrightarrow \cat{A} & F^*(a,b) &= \cat{B}(Fa,b).
\end{aligned}\]

There is a bicategory $\cat{V}\hy\Prof$ of profunctors introduced by B\'{e}nabou~\cite{Benabou1973} whose 0-cells are small categories, and where $\cat{V}\hy\Prof(\cat{A},\cat{B})$ is the ordinary category $[\cat{B}^\opp \otimes \cat{A},\cat{V}]_0$ of $\cat{V}$-functors and $\cat{V}$-natural transformations. The composite of two profunctors $\phi \colon \cat{A} \slashedrightarrow \cat{B}$ and $\psi \colon \cat{B} \slashedrightarrow \cat{C}$, denoted by $\psi \otimes \phi$, is given by the coend
\[(\psi \otimes \phi)(c,a) = \int^b \phi(b,a) \otimes \psi(c,b).\]
(Note the reversal of the order in which $\psi$ and $\phi$ appear on either side of this definition.) The identity profunctor on $\cat{A}$ is simply $\cat{A}(-,-)$. The fact that this is an identity for composition of profunctors follows from the density formula.

With this definition, the assignments $F \mapsto F_*$ and $F \mapsto F^*$ become bifunctors 
\[(-)_* \colon \cat{V}\hy\Cat \to \cat{V}\hy\Prof \qquad \text{and} \qquad (-)^* \colon \cat{V}\hy\Cat^{\coo\opp} \to \cat{V}\hy\Prof.\]
The main result about $\cat{V}\hy\Prof$ is that $F_* \dashv F^*$ for any functor $F \colon \cat{A} \to \cat{B}$. The unit and counit of this adjunction have components
\begin{align*}
    \eta_{a,a'} &\colon \cat{A}(a,a') \to \int^b \cat{B}(b,Fa') \otimes \cat{B}(Fa,b) = \cat{B}(Fa,Fa') \\
    \eps_{b,b'} &\colon \int^a \cat{B}(Fa,b') \otimes \cat{B}(b,Fa) \to \cat{B}(b,b'),
\end{align*}
where $\eta_{a,a'}$ is the action of $F$ on morphisms, and where $\eps_{b,b'}$ is induced by the composition of $\cat{B}$. With this description, it is clear that $F$ is fully faithful if and only if $\eta$ is an isomorphism. Being Cauchy dense is the dual condition:

\begin{definition}\label{def:Cauchy_dense}
    A functor $F \colon \cat{A} \to \cat{B}$ between small categories is \demph{Cauchy dense} if the counit $\eps$ of the profunctor adjunction $F_* \dashv F^*$ is an isomorphism.
\end{definition}

Lawvere first introduced this condition in~\cite[p.~155]{Lawvere1973}, where he calls it \emph{$\cat{V}$-density}. That term has since become standard for the (as Theorem~\ref{thm:abs_dense} shows) weaker condition that $\cat{B}(F-,-) \colon \cat{B} \to [\cat{A}^\opp,\cat{V}]$ be fully faithful, or equivalently that $\Lan_F F = 1_\cat{B}$. The term `Cauchy dense' used in this note is taken from Day~\cite[p.~428]{Day1977}, and it is motivated by the $\cat{V} = \R^+$ case, as the next proposition shows.

\begin{proposition}\label{prop:top_density1}
    An $\R^+$-functor $f \colon A \to B$ between classical metric spaces is Cauchy dense iff its image is topologically dense in $B$.
\end{proposition}
\begin{proof}
    The Cauchy density of $f$ reduces to the condition
    \[\inf_a (B(b, fa) + B(fa,b')) = B(b,b') \qquad \forall b, b' \in B.\]
    Taking $b' = b$, we see that for any $\eps > 0$ there is some $a \in A$ with $B(b,fa) < \eps$. Hence, $fA$ is dense in $B$.

    Conversely, let $b, b' \in B$ and, for any $\eps > 0$, take $a \in A$ with $B(b,fa) < \eps$. Then
    \[B(b,fa) + B(fa,b') \leq B(b,fa) + B(fa,b) + B(b,b') < 2\eps + B(b,b').\]
    Since $\eps$ can be made arbitrarily small, we conclude that $f$ is Cauchy dense.
\end{proof}

For more general $\cat{V}$, the Cauchy density of $F \colon \cat{A} \to \cat{B}$ does not only depend on its image, although it does if $\cat{V}_0$ is a preorder (see Proposition~\ref{prop:V_preorder}).

Just as being fully faithful is a self-dual condition, so is Cauchy density.

\begin{proposition}\label{prop:self-dual}
    A functor $F \colon \cat{A} \to \cat{B}$ is Cauchy dense iff $F^\opp \colon \cat{A}^\opp \to \cat{B}^\opp$ is.
\end{proposition}
\begin{proof}
    For all $a \in \cat{A}$ and $b,b' \in \cat{B}$, consider the commutative diagram
    \[\begin{tikzcd}
    \cat{B}(Fa,b') \otimes \cat{B}(b,Fa) \ar[r, "M_\cat{B}"] \ar[d, "\sigma"', "\cong"] & \cat{B}(b,b') \ar[d, equals] \\
    \cat{B}^\opp(Fa,b) \otimes \cat{B}^\opp(b',Fa) \ar[r, "M_{\cat{B}^\opp}"'] & \cat{B}^\opp(b',b),
    \end{tikzcd}\]
    where $\sigma$ is the symmetry of $\cat{V}$ and $M_\cat{B}$ and $M_{\cat{B}^\opp}$ are the composition of the respective categories. The counit $\eps^F$ of $F_* \dashv F^*$ is the map induced by the top arrow as $a$ varies, while the counit $\eps^{F^\opp}$ of $(F^\opp)_* \dashv (F^\opp)^*$ is the map induced by the bottom arrow. Clearly, $\eps^F$ is an isomorphism iff $\eps^{F^\opp}$ is.
\end{proof}

Since bifunctors preserve adjunctions, it is easy to say when a left adjoint is Cauchy dense. In fact, thanks to the self-duality of Cauchy density, the same applies to right adjoints:

\begin{proposition}\label{prop:Cauchy_dense_adjoint}
    A functor with an adjoint (either left or right) is Cauchy dense iff said adjoint is fully faithful.
\end{proposition}
\begin{proof}
    If $F \colon \cat{A} \to \cat{B}$ has a right adjoint $G$, then $F^* \cong G_*$ and the counit of the profunctor adjunction $F_* \dashv G_*$ is given by $\cat{B}(-,\eps_{=}) \colon \cat{B}(-,FG{=}) \to \cat{B}(-,{=})$, where $\eps$ is the counit of $F \dashv G$. This is an isomorphism iff $\eps$ is an isomorphism iff $G$ is fully faithful. If $F$ has a left adjoint, we apply this argument to $F^\opp$.
\end{proof}

This proposition gives yet another proof of the folklore fact that, if $H \dashv F \dashv G$ is an adjoint triple of functors, then $H$ is fully faithful iff $G$ is.

\begin{example}\label{ex:CD_adjoint}
    Any reflection and coreflection is Cauchy dense. Some examples are the free functors $\Mon \to \Grp$ and $\Grp \to \Ab$, the forgetful functor $\Top \to \Set$, and the completion functor from $\Met$ to its full subcategory on the complete metric spaces. These categories are not small, but one can make sense of these statements using the contents of Subsection~\ref{ssec:large}.
\end{example}

Cauchy dense functors were identified by Lucatelli Nunes and Sousa~\cite{LucatelliNunes2022} as the \demph{lax epimorphisms} in the 2-category $\cat{V}\hy\Cat$, i.e.\ the $\cat{V}$-functors $F \colon \cat{A} \to \cat{B}$ such that the ordinary functor $[F,\cat{C}]_0 \colon [\cat{B},\cat{C}]_0 \to [\cat{A},\cat{C}]_0$ is fully faithful for all small categories $\cat{C}$. The following theorem follows from their Theorems~5.6 and~5.11.

Recall that a functor $F \colon \cat{A} \to \cat{B}$ is dense if $\Lan_F F = 1_\cat{B}$. It is \demph{absolutely dense} if, moreover, $\Lan_F F$ is preserved by any functor $G \colon \cat{B} \to \cat{C}$. Similarly, $F$ is \demph{absolutely codense} if $\Ran_F F = 1_\cat{B}$ and this Kan extension is preserved by any functor.

\begin{theorem}[Lucatelli Nunes \& Sousa]\label{thm:abs_dense}
    The following are equivalent for a functor $F \colon \cat{A} \to \cat{B}$ between small categories:
    \begin{enumerate}[label={\upshape(\alph*)}]
        \item\label{part:lax_epi} $F$ is a lax epimorphism in $\cat{V}\hy\Cat$;
        \item\label{part:V-lax_epi} $[F,\cat{C}] \colon [\cat{B},\cat{C}] \to [\cat{A},\cat{C}]$ is a fully faithful $\cat{V}$-functor for all small $\cat{C}$;
        \item\label{part:Lan} there is an isomorphism $\Lan_F \cat{B}(b,F-) \cong \cat{B}(b,{=})$ natural in $b \in \cat{B}^\opp$;
        \item\label{part:CD} $F$ is Cauchy dense;
        \item\label{part:abs_dense} $F$ is absolutely dense;
        \item $F$ is absolutely codense.
    \end{enumerate}  
\end{theorem}
\begin{proof}
    The only new condition here is~\ref{part:CD}. It implies~\ref{part:Lan} and is implied by~\ref{part:abs_dense}, since
    \[\Lan_F \cat{B}(b,F-) = (F_* \otimes F^*)(b,{=})\]
    and the canonical map $\Lan_F \cat{B}(b,F-) \to \cat{B}(b,{=})$ is the one induced by the composition of $\cat{B}$.
\end{proof}

Conditions~\ref{part:lax_epi} and~\ref{part:V-lax_epi} are reminiscent of the fact that a short map, i.e.\ an $\R^+$-functor, $g \colon B \to C$ between classical metric spaces is determined by its values on any dense subset of $B$. Hence, if a short map $f \colon A \to B$ has a dense image, $g$ is determined by $gf$. In other words, $\Met(f,C) \colon \Met(B,C) \to \Met(A,C)$ is injective. If, in addition, $f$ is an isometric embedding (i.e.\ fully faithful) and $C$ is complete, then $\Met(f,C)$ is a bijection. The analogue of this will turn out to be true in general for fully faithful Cauchy dense functors -- see Theorem~\ref{thm:ffCD_iff_pre_equiv}.

\begin{corollary}\label{cor:LA_CD_iff}
    Let $F \colon \cat{A} \to \cat{B}$ be left adjoint to $G$. Then the following conditions are equivalent:
    \begin{enumerate}[label={\upshape(\alph*)}]
        \item\label{part:G_ff} $G$ is fully faithful;
        \item\label{part:F_dense} $F$ is dense;
        \item $F$ is absolutely dense;
        \item $F$ is absolutely codense;
        \item $F$ is Cauchy dense.
    \end{enumerate}
\end{corollary}
\begin{proof}
    The counit of the Kan extension $\Lan_F F = FG \to 1_\cat{B}$ is the counit of the adjunction, so~\ref{part:G_ff} and~\ref{part:F_dense} are equivalent. The remaining conditions are equivalent to~\ref{part:G_ff} by Proposition~\ref{prop:Cauchy_dense_adjoint} and Theorem~\ref{thm:abs_dense}.
\end{proof}

Unlike (co)dense functors, Cauchy dense functors are closed under composition. Parts~\ref{part:F_G_cd} and~\ref{part:GF_F_cd} of the next proposition also follow from the characterisation of Cauchy dense functors as lax epimorphisms in $\cat{V}\hy\Cat$ (see~\cite[Rmk.~2.3]{LucatelliNunes2022}).

\begin{proposition}\label{prop:closure}
    Let $F \colon \cat{A} \to \cat{B}$ and $G \colon \cat{B} \to \cat{C}$ be functors between small categories.
    \begin{enumerate}[label={\upshape(\alph*)}]
        \item\label{part:F_G_cd} If $F$ and $G$ are Cauchy dense, then so is $GF$.
        \item\label{part:GF_cd_G_ff} If $GF$ is Cauchy dense and $G$ is fully faithful, then $F$ is Cauchy dense.
        \item\label{part:GF_F_cd} If $GF$ is Cauchy dense and $F$ is Cauchy dense, then $G$ is Cauchy dense.
    \end{enumerate}
\end{proposition}
\begin{proof}
    The counit $\eps^{GF}$ of the profunctor adjunction $(GF)_* \dashv (GF)^*$ is the composite
    \[\begin{tikzcd}
    G_* \otimes F_* \otimes F^* \otimes G^* \ar[r, "G_* \otimes \eps^F \otimes G^*"] &[3em] G_* \otimes G^* \ar[r, "\eps^G"] & 1_{\cat{C}},
    \end{tikzcd}\]
    where $\eps^F$ and $\eps^G$ are the counits of $F_*\dashv F^*$ and $G_* \dashv G^*$, respectively.

    If $\eps^F$ and $\eps^G$ are isomorphisms, then so is $\eps^{GF}$, giving~\ref{part:F_G_cd}. Similarly, if $\eps^F$ and $\eps^{GF}$ are isomorphisms, then so is $\eps^G$, giving~\ref{part:GF_F_cd}.

    For~\ref{part:GF_cd_G_ff}, note that for all $b,b' \in \cat{B}$ we have a commutative square
    \[\begin{tikzcd}
    \int^a \cat{B}(Fa,b') \otimes \cat{B}(b,Fa) \ar[r, "\eps^F_{b,b'}"] \ar[d, "\int^a G \otimes G"'] &[2em] \cat{B}(b,b') \ar[d, "G"] \\
    \int^a \cat{B}(GFa, Gb') \otimes \cat{B}(Gb, GFa) \ar[r, "\eps^{GF}_{Gb,Gb'}"'] & \cat{C}(Gb,Gb').
    \end{tikzcd}\]
    If $G$ is fully faithful, then the vertical arrows are isomorphisms. Hence, if $\eps^{GF}$ is an isomorphism, so is $\eps^F$.
\end{proof}

\begin{remark}
    If one removes from~\ref{part:GF_cd_G_ff} the assumption that $G$ is fully faithful, then the result is no longer true, even if full faithfulness is replaced by Cauchy density. An example can be easily found among functors between groups, since for such functors Cauchy density is equivalent to surjectivity (i.e.\ fullness) -- see Corollary~\ref{cor:group_CD}.

    If one removes from~\ref{part:GF_F_cd} the assumption that $F$ is Cauchy dense, then the result is no longer true. A counterexample can easily be found among functors between discrete categories, since for such functors Cauchy density is equivalent to bijectivity (on objects) -- see Example~\ref{ex:discrete_CD}.
\end{remark}

The previous proposition shows that fully faithful Cauchy dense functors enjoy strong closure properties. Such functors will be the subject of Section~\ref{sec:ffCD}.

Any lax monoidal functor $S \colon \cat{V}_0 \to \cat{W}_0$, gives a 2-functor $S_* \colon \cat{V}\hy\CAT \to \cat{W}\hy\CAT$, often called \demph{base change}. The next proposition gives some elementary results on how Cauchy density behaves under base change, at least when $S$ is strong monoidal.

\begin{proposition}\label{prop:base_change}
    Let $S \colon \cat{V}_0 \to \cat{W}_0$ be a strong monoidal functor, and $F \colon \cat{A} \to \cat{B}$ be a $\cat{V}$-functor.
    \begin{enumerate}[label={\upshape(\alph*)}]
        \item \label{part:cocts_CD} If $S$ is cocontinuous and $F$ is Cauchy dense, then so is $S_* F$.
        \item \label{part:recolim_CD} If $S$ reflects colimits and $S_*F$ is Cauchy dense, then so is $F$.
    \end{enumerate}
\end{proposition}
\begin{proof}
    For each $b,b' \in \cat{B}$ and $a \in \cat{A}$, we have a commutative triangle in $\cat{W}_0$
    \[\begin{tikzcd}
    S\cat{B}(Fa,b') \otimes S\cat{B}(b,Fa) \ar[dr, "M_{S_*\cat{B}}"] \ar[d, "\mu"', "\cong"] \\
    S(\cat{B}(Fa,b') \otimes \cat{B}(b,Fa)) \ar[r, "SM_\cat{B}"'] & S\cat{B}(b,b'),
    \end{tikzcd}\]
    where $\mu$ is the coherence isomorphism of $S$. For~\ref{part:cocts_CD}, the cowedge $M_\cat{B}$ is a coend cowedge. Since $S$ is cocontinuous, so is $SM_\cat{B}$ and $M_{S_*\cat{B}}$. For~\ref{part:recolim_CD}, the cowedge $M_{S_*\cat{B}}$ is a coend cowedge, and then so is $SM_\cat{B}$. Since $S$ reflects colimits, $M_\cat{B}$ is also a coend cowedge.
\end{proof}

\begin{example}
    Let $\mathsf{2} = \{\bot < \top\}$ be the walking arrow category. There is a strong monoidal functor $S \colon \mathsf{2} \to \Set$ which sends $\bot \mapsto \varnothing$ and $\top \to 1$. It reflects colimits, being fully faithful, so an order preserving map $f \colon A \to B$ between two $\mathsf{2}$-categories is Cauchy dense if the corresponding functor $S_*f$ between preorders is Cauchy dense. The converse is not true -- see Subsection~\ref{ssec:preorders}.
\end{example}

\subsection{Between large categories}\label{ssec:large}

We conclude this section with some generalisations to large categories. We will say that $F \colon \cat{A} \to \cat{B}$ is \demph{Cauchy dense} if the coend
\[\int^a \cat{B}(Fa,b') \otimes \cat{B}(b,Fa)\]
exists and the map from it to $\cat{B}(b,b')$ induced by the composition of $\cat{B}$ is an isomorphism for all $b,b' \in \cat{B}$.

The proof of Proposition~\ref{prop:self-dual} still applies to this new definition, so that $F$ is Cauchy dense iff $F^\opp$ is. Note also that this coend is the pointwise formula for $\Lan_F \cat{B}(b,F-)$ evaluated at $b'$, so that its existence for all $b'$ gives the pointwise existence of this left Kan extension.

In general, when $\cat{A}$ is large, there is no $\cat{V}$-category $[\cat{A},\cat{B}]$, because defining the $\cat{V}$-object of $\cat{V}$-natural transformations between two functors will involve a large end. However, one can still generalise the parts of Theorem~\ref{thm:abs_dense} that do not involve functor categories.

\begin{theorem}\label{thm:large_iff}
    The following are equivalent for a functor $F \colon \cat{A} \to \cat{B}$ between (not necessarily small) categories:
    \begin{enumerate}[label={\upshape(\alph*)}]
        \item $F$\label{part:large_F_CD} is Cauchy dense;
        \item $F$\label{part:large_F_dense} is absolutely dense;
        \item $F$\label{part:large_F_codense} is absolutely codense.
    \end{enumerate}
\end{theorem}
\begin{proof}
    We show that~\ref{part:large_F_CD} and~\ref{part:large_F_codense} are equivalent. The equivalence of~\ref{part:large_F_CD} and~\ref{part:large_F_dense} will then follow from the fact that Cauchy density is a self-dual condition.

    Let $F$ be Cauchy dense and $H \colon \cat{B} \to \cat{C}$ be an arbitrary functor. We must show that $Hb$ satisfies the same universal property of $\Ran_F HF$ evaluated at $b$ for all $b \in \cat{B}$. Let $c \in \cat{C}$. We have a series of natural isomorphisms
    \begin{align*}
        \cat{C}(c, Hb) &\cong [\cat{B},\cat{V}](\cat{B}(b,-), \cat{C}(c,H-)) && \mbox{by Yoneda,}\\
        &\cong [\cat{B},\cat{V}](\Lan_F \cat{B}(b,F-), \cat{C}(c,H-)) && \mbox{since $F$ is Cauchy dense,} \\
        &\cong [\cat{A},\cat{V}](\cat{B}(b,F-),\cat{C}(c,HF-)) && \mbox{by the universal property of $\Lan$.}
    \end{align*}
    Note that, a priori, only the first object of $\cat{V}$ listed exists, since the rest require some weighted limit in $\cat{V}$ indexed by a large category. However, the series of isomorphisms shows that, in fact, all of them exist. This proves that $Hb$ satisfies the defining universal property of $\{\cat{B}(b,F-), HF\}$, which gives the value at $b$ of $\Ran_F HF$.

    Conversely, let $F$ be absolutely codense. Then $F^\opp$ is absolutely dense, which implies that $\Lan_{F^\opp} \cat{B}(F-,b') = \cat{B}(-,b')$ for all $b' \in \cat{B}$. This is equivalent to saying that $F$ is Cauchy dense.
\end{proof}

With this, Proposition~\ref{prop:Cauchy_dense_adjoint} and Corollary~\ref{cor:LA_CD_iff} remain true for functors between large categories. This justifies the claims made in Example~\ref{ex:CD_adjoint}, since if $F$ has a left or right adjoint the relevant coends always exists.

\begin{remark}\label{rmk:large_closure}
With this definition of Cauchy density of functors between large categories, one can try to remove the hypothesis in Proposition~\ref{prop:closure} that $\cat{A}$, $\cat{B}$ and $\cat{C}$ are small. Part~\ref{part:GF_cd_G_ff} is still true, with the same proof. Parts~\ref{part:F_G_cd} and~\ref{part:GF_F_cd} hold as soon as $\cat{A}$ and $\cat{B}$ are small.
\end{remark}

\section{Fully faithful Cauchy dense functors}\label{sec:ffCD}

Given a metric space $A$, its completion $\overline{A}$ is the largest space into which $A$ embeds as a dense subset. In our language, this means that if $f \colon A \to B$ is a fully faithful Cauchy dense $\R^+$-functor, then $B$ embeds into $\overline{A}$. The main purpose of this section is to generalise this fact to $\cat{V}$-categories.

Given a small category $\cat{A}$ we write $\overline{\cat{A}}$ for its \demph{Cauchy completion}. It is the full subcategory of $[\cat{A}^\opp,\cat{V}]$ on those presheaves that satisfy any of the equivalent conditions in the next theorem, a proof of which can be assembled from the contents of Kelly and Schmitt~\cite[\S6]{Kelly2005}, where they also cite Street~\cite{Street1983}.

\begin{theorem}[Kelly, Schmitt, Street]\label{thm:Ccomp}
Let $\cat{A}$ be a small category. For a functor $W \colon \cat{A}^\opp \to \cat{V}$, the following are equivalent:
\begin{enumerate}[label={\upshape(\alph*)}]
    \item $W$ is small-projective, i.e.\ $[\cat{A}^\opp,\cat{V}](W,-)$ preserves all small colimits;
    \item $W$ is an absolute weight, i.e.\ $F$-weighted colimits are preserved by any functor;
    \item $W$ has a right adjoint in $\cat{V}\hy\Prof$ when seen as profunctor $\mathsf{1} \slashedrightarrow \cat{A}$.
\end{enumerate}
\end{theorem}

The category $\cat{A}$ is \demph{Cauchy complete} if $\overline{\cat{A}}$ coincides with $\cat{A}$, or rather with its image under the Yoneda embedding. Equivalently, $\cat{A}$ is Cauchy complete if it has all absolute colimits. The name `Cauchy completion' was suggested by Lawvere~\cite[\S3]{Lawvere1973}, motivated by the fact that when $A$ is a metric space, its Cauchy completion $\overline{A}$ in the $\R^+$-enriched sense can be identified with the familiar metric completion of $A$ (the space of equivalence classes of Cauchy sequences in $A$).

Our main theorem relates fully faithful Cauchy dense functors and Cauchy completions. The key result that allows this, Lemma~\ref{lem:sfCD_Cauchy_completion} below, actually holds in slightly more generality, for split-full Cauchy dense functors.

\begin{definition}
    A functor $F \colon \cat{A} \to \cat{B}$ is \demph{split-full} if $F_{a,a'} \colon \cat{A}(a,a') \to \cat{B}(Fa,Fa')$ is a split epimorphism for all $a,a' \in \cat{A}$.
\end{definition}

Assuming the axiom of choice, every full functor over $\cat{V} = \Set$ is split-full. For different $\cat{V}$, however, this is stronger than merely requiring $F_{a,a'}$ to be epic. For example, if $\cat{V}_0$ is a preorder, then a $\cat{V}$-functor is split-full iff it is fully faithful.

\begin{lemma}\label{lem:sfCD_Cauchy_completion}
    Let $\cat{A}$ be a small category and $F \colon \cat{A} \to \cat{B}$ be a split-full Cauchy dense functor. Then $\cat{B}(F-,b) \colon \cat{A}^\opp \to \cat{V}$ lies in $\overline{\cat{A}}$ for all $b \in \cat{B}$.
\end{lemma}
\begin{proof}
    We will show that $\cat{B}(F-,b) \colon \mathsf{1} \slashedrightarrow \cat{A}$ is left adjoint to $\cat{B}(b,F-) \colon \cat{A} \slashedrightarrow \mathsf{1}$ in $\cat{V}\hy\Prof$. The unit is a morphism $I \to \int^a \cat{B}(Fa,b) \otimes \cat{B}(b,Fa)$ in $\cat{V}$, which we take to be the composite
    \[\delta_* = \begin{tikzcd}
    I \ar[r, "j_b"] & \cat{B}(b,b) \ar[r, "\eps_{b,b}^{-1}"] & \int^a \cat{B}(Fa,b) \otimes \cat{B}(b,Fa).
    \end{tikzcd}\]
    The counit has a component $\nu_{a,a'} \colon \cat{B}(b,Fa') \otimes \cat{B}(Fa,b) \to \cat{A}(a,a')$ for each $a,a' \in \cat{A}$, which we take to be the composite
    \[\nu_{a,a'} = \begin{tikzcd}
    \cat{B}(b,Fa') \otimes \cat{B}(Fa,b) \ar[r, "M_\cat{B}"] & \cat{B}(Fa,Fa') \ar[r, "F_{a,a'}^r"] & \cat{A}(a,a'),
    \end{tikzcd}\]
    where $F^{r}_{a,a'}$ is a right inverse to $F_{a,a'}$.

    We check one of the triangle identities, the other being exactly dual. For each $a \in \cat{A}$, consider the following diagram in $\cat{V}$:
    \[\begin{tikzcd}
    I \otimes \cat{B}(Fa,b) \ar[r, "j_b \otimes 1"] \ar[dr, "l"']
    &[-5pt] \cat{B}(b,b) \otimes \cat{B}(Fa,b) \ar[r, "\eps^{-1}_{b,b} \otimes 1"] \ar[d, "M_\cat{B}"]
    &[-5pt] \int^{a'} \cat{B}(Fa',b) \otimes \cat{B}(b,Fa') \otimes \cat{B}(Fa,b) \ar[d, "\int 1 \otimes M_{\cat{B}}"]\\
    & \cat{B}(Fa,b) \ar[dr, "i^{-1}"']
    & \int^{a'} \cat{B}(Fa',b) \otimes \cat{B}(Fa,Fa') \ar[d, "\int 1 \otimes F^{r}_{a,a'}"] \ar[l, "\eps_{Fa,b}"']\\
    && \int^{a'} \cat{B}(Fa',b) \otimes \cat{A}(a,a')
    \end{tikzcd}\]
    The arrow labelled $l$ is the left unitor in $\cat{V}$, and $i$ is the isomorphism in the density formula.
    
    The triangle on the left commutes by unitality in $\cat{B}$. The square commutes since the same square with the top arrow replaced by its inverse commutes, by associativity in $\cat{B}$.
    
    It remains to check the commutativity of the bottom triangle. The density formula implies that $\eps_{Fa,b} \circ (\int 1 \otimes F_{a,a'}) = i$. It follows that $i^{-1} \circ \eps_{Fa,b} = \int 1 \otimes F^r_{a,a'}$, as needed.
    
    This shows that $(\nu \otimes \cat{B}(F-,b)) \circ (\cat{B}(F-,b) \otimes \delta)$ is the identity on $\cat{B}(F-,b)$, thus establishing one of the triangle identities.
\end{proof}

It is no coincidence that in the previous proof we have only used that $\eps_{b,b}$ is invertible for all $b \in \cat{B}$, rather than the full strength of Cauchy density, i.e.\ that $\eps_{b,b'}$ is invertible for all $b,b'\in \cat{B}$. For split-full functors, Cauchy density is equivalent to this apparently weaker condition -- see Theorem~\ref{thm:split_full_CD}.

\begin{remark}
    If one strengthens the hypotheses of Lemma~\ref{lem:sfCD_Cauchy_completion} so that $F$ is fully faithful and Cauchy dense, then a simpler proof is possible through Theorem~\ref{thm:large_iff}, Avery and Leinster~\cite[Lem.~8.1]{Avery2021} and Kelly and Schmitt~\cite[Prop.~6.14]{Kelly2005}. In fact, from our lemma and the latter result, it follows that if $F$ is split-full and Cauchy dense, then $\cat{B}(F-,b)$ and $\cat{B}(b,F-)$ are a pair of Isbell conjugate presheaves.
\end{remark}

We are now ready to prove the universal property of the Cauchy completion with respect to fully faithful Cauchy dense functors.

\begin{theorem}\label{thm:ffCD}
    Let $\cat{A}$ be a small category. Then $\overline{\cat{A}}$ is the largest category that admits a fully faithful Cauchy dense functor from $\cat{A}$, i.e.\ for every fully faithful Cauchy dense functor $F \colon \cat{A} \to \cat{B}$ there exists an essentially unique functor (necessarily fully faithful) $N_F$ such that the triangle
    \[\begin{tikzcd}[column sep=tiny]
    \cat{A} \ar[dr, "F"', hook] \ar[rr, "z", hook] && \overline{\cat{A}} \\
    & \cat{B} \ar[ur, "N_F"', hook]
    \end{tikzcd}\]
    commutes up to isomorphism, where $z$ is the factorisation of the Yoneda embedding. Moreover, the inclusion of $\cat{A}$ into any full subcategory of $\overline{\cat{A}}$ containing the representables is Cauchy dense.
\end{theorem}
\begin{proof}
    We take $N_F(b) = \cat{B}(F-,b)$, which by Lemma~\ref{lem:sfCD_Cauchy_completion} lies in $\overline{\cat{A}}$. Then $N_F(Fa) = \cat{B}(F-,Fa) \cong \cat{A}(-,a) = za$ naturally in $a \in \cat{A}$ because $F$ is fully faithful, so the triangle in the statement commutes up to isomorphism. By Theorem~\ref{thm:abs_dense}, $F$ is dense, which is to say that $N_F$ is fully faithful. 
    
    Since $N_F$ reflects colimits, each object of $\cat{B}$ is an absolute colimit of objects in the image of $F$. Any other functor $N'_F \colon \cat{B} \to \overline{\cat{A}}$ completing the triangle in the statement preserves these colimits, and hence must be isomorphic to $N_F$.

    For the last statement, it suffices by Proposition~\ref{prop:closure}\ref{part:GF_cd_G_ff} and Remark~\ref{rmk:large_closure} to show that $z$ is Cauchy dense. It is so, by Theorem~\ref{thm:large_iff}, because it is absolutely dense by the definition of $\overline{\cat{A}}$.
\end{proof}

Hence, a fully faithful Cauchy dense functor amounts to the inclusion of a category $\cat{A}$ into a copy of itself with freely adjoined colimits for some specified set of diagrams with absolute weights.

Avery and Leinster~\cite[Cor.~8.3 and Thm.~8.4]{Avery2021} showed that the reflexive completion $\cat{R}(\cat{A})$ of a small category $\cat{A}$ is the largest category into which $\cat{A}$ admits an adequate functor, i.e.\ a functor $F \colon \cat{A} \to \cat{B}$ which is fully faithful, dense and codense. Theorem~\ref{thm:ffCD} shows that the Cauchy completion admits a similar description, where adequate functors are replaced by fully faithful Cauchy dense functors. Since every Cauchy dense functor is dense and codense, it follows that $\overline{\cat{A}} \subseteq \cat{R}(\cat{A})$, as Avery and Leinster have already shown~\cite[Prop.~10.1]{Avery2021}.

\begin{corollary}\label{cor:eq_CC}
    Two small categories $\cat{A}$ and $\cat{B}$ are Morita equivalent, i.e.\ $\overline{\cat{A}} \simeq \overline{\cat{B}}$, iff there is a zigzag of fully faithful Cauchy dense functors connecting them.
\end{corollary}
\begin{proof}
    If $\cat{A}$ and $\cat{B}$ are Morita equivalent, then $\begin{tikzcd}[cramped, column sep=scriptsize] \cat{A} \ar[r, "z_\cat{A}"] & \overline{\cat{A}} \simeq \overline{\cat{B}} & \cat{B} \ar[l, "z_\cat{B}"']\end{tikzcd}$ gives the desired zigzag.

    Conversely, it suffices to show that if $F \colon \cat{A} \to \cat{B}$ is fully faithful and Cauchy dense, then $\overline{\cat{A}} \simeq \overline{\cat{B}}$. Consider the following diagram of fully faithful functors:
    \[\begin{tikzcd}
    \cat{A} \ar[d, "F"', hook] \ar[r, "z_\cat{A}", hook] & \overline{\cat{A}} \ar[d, "N_{N_F}", hook] \\
    \cat{B} \ar[ur, "N_F", hook] \ar[r, "z_\cat{B}", hook] & \overline{\cat{B}}.
    \ar[u, bend right=75, looseness=1.5, "N_{z_\cat{B} F}"', hook]
    \end{tikzcd}\]
    Note that $N_F$ and $z_\cat{B} F$ are fully faithful and Cauchy dense by Proposition~\ref{prop:closure}\ref{part:GF_cd_G_ff} and Remark~\ref{rmk:large_closure}, so $N_{N_F}$ and $N_{z_\cat{B}F}$ exist by Theorem~\ref{thm:ffCD}. Then $N_{z_\cat{B}F} N_{N_F}$ and $N_{N_F} N_{z_\cat{B}F}$ are isomorphic to the respective identities, because they are so on representables and Cauchy completion is the free cocompletion with respect to absolute colimits.
\end{proof}

The next theorem gives an alternative way of identifying the fully faithful Cauchy dense functors among the Cauchy dense ones, in view of Theorem~\ref{thm:abs_dense}. 

\begin{theorem}\label{thm:ffCD_iff_pre_equiv}
    Let $F \colon \cat{A} \to \cat{B}$ be a functor between small categories. Then $F$ is fully faithful and Cauchy dense iff the functor $[F,\cat{C}] \colon [\cat{B},\cat{C}] \to [\cat{A},\cat{C}]$ is an equivalence for every Cauchy complete category $\cat{C}$.
\end{theorem}
\begin{proof}
    First assume that $F$ is fully faithful and Cauchy dense. In the commutative diagram
    \[\begin{tikzcd}[column sep=3.5em]
    [\overline{\cat{B}},\cat{C}] \ar[d, "{[N_{N_F},\cat{C}]}"'] \ar[r, "{[z_\cat{B},\cat{C}]}"] &{} [\cat{B},\cat{C}] \ar[d, "{[F,\cat{C}]}"] \\{}
    [\overline{\cat{A}},\cat{C}] \ar[r, "{[z_\cat{A},\cat{C}]}"] &{} [\cat{A},\cat{C}],
    \end{tikzcd}\]
    the horizontal arrows are equivalences by Kelly and Schmitt~\cite[7.1]{Kelly2005}, and so is the left one by the proof of Corollary~\ref{cor:eq_CC}. Therefore, $[F,\cat{C}]$ must be an equivalence too.

    Conversely, consider the commutative diagram
    \[\begin{tikzcd}{}
    \cat{A} \ar[r, "y_\cat{A}"] \ar[d, "F"'] &{} [\cat{A}^\opp,\cat{V}] \ar[d, "\widehat{F}"] \\
    \cat{B} \ar[r, "y_\cat{B}"'] &{} [\cat{B}^\opp,\cat{V}],
    \end{tikzcd}\]
    where $\widehat{F} = \Lan_{y_\cat{A}} y_\cat{B} F$ is the essentially unique cocontinuous functor that makes the diagram commute. In such cases, $\widehat{F}$ always has a right adjoint, which sends $P \colon \cat{B}^\opp \to \cat{V}$ to $[\cat{B}^\opp,\cat{V}](y_\cat{B}F-,P)$ -- see e.g. Kelly~\cite[Thm~4.51]{Kelly1982}. By Yoneda, this right adjoint is just $[F^\opp,\cat{V}]$. Since $\cat{V}$ is complete, $\cat{V}^\opp$ is Cauchy complete, so $[F,\cat{V}^\opp]$ is an equivalence, and then so is $[F,\cat{V}^\opp]^\opp = [F^\opp,\cat{V}]$. Its left adjoint, $\widehat{F}$, must also be an equivalence. It follows that $F$ is fully faithful.
    
    Moreover, $\widehat{F}$ restricts to an equivalence between the full-subcategories of small-projective objects, i.e.\ an equivalence $\overline{F} \colon \overline{\cat{A}} \simeq \overline{\cat{B}}$. Since $\overline{F} z_\cat{A} \cong z_\cat{B} F$ is Cauchy dense and $z_\cat{B}$ is fully faithful, Proposition~\ref{prop:closure}\ref{part:GF_cd_G_ff} gives that $F$ is Cauchy dense.
\end{proof}

\section{Examples}\label{sec:ex}

This last section deals with examples and characterisations of Cauchy dense functors in different contexts.

\subsection{When \texorpdfstring{$\cat{V}_0$}{V0} is a preorder}\label{ssec:V-preorder}

In Proposition~\ref{prop:top_density1}, we saw that for metric spaces Cauchy density reduces to topological density. This condition quantifies over each point of the codomain, instead of over each pair of points of the codomain, as Cauchy density does. If $\cat{V}_0$ is a preorder, it turns out that one can always make this simplification.

\begin{proposition}\label{prop:V_preorder}
    Let $\cat{V}_0$ be a preorder. A $\cat{V}$-functor $F \colon \cat{A} \to \cat{B}$ is Cauchy dense iff $\eps_{b,b} \colon (F_* \otimes F^*)(b,b) \to \cat{B}(b,b)$ is an isomorphism for all $b \in \cat{B}$. Moreover, whether $F$ is Cauchy dense only depends on the image of its object mapping.
\end{proposition}
\begin{proof}
    The forward implication is clear. Now assume that $\eps_{b,b}$ is an isomorphism for all $b \in \cat{B}$. Let $b, b' \in \cat{B}$ and consider the following diagram in $\cat{V}_0$
    \[\begin{tikzcd}
    \int^a \cat{B}(Fa,b') \otimes \cat{B}(b,Fa) 
    \ar[r, "\eps_{b,b'}"] 
    & \cat{B}(b,b') 
    \ar[d, "1 \otimes j_b"] \\
    \int^a \cat{B}(b,b') \otimes \cat{B}(Fa,b) \otimes \cat{B}(b,Fa) 
    \ar[u, "\int^a M_\cat{B} \otimes 1"] \ar[r, "1 \otimes \eps_{b,b}"'] 
    & \cat{B}(b,b') \otimes \cat{B}(b,b).
    \end{tikzcd}\]
    By assumption, the bottom arrow has an inverse, so we can compose with it to get a morphism in the opposite direction to $\eps_{b,b'}$. Since $\cat{V}_0$ is a preorder, this suffices to show that $\eps_{b,b'}$ is an isomorphism.

    For the last statement, note that since $\cat{V}_0$ is a preorder, coends are simply suprema. Hence, whether $\eps_{b,b'}$ is an isomorphism only depends on the action of $F$ on objects.
\end{proof}

\begin{remark}
    This is manifestly not true when $\cat{V}_0$ is not a preorder. The diagram in the previous proof does not typically commute. For example, the identity-on-objects functor from the discrete category $\{\bot,\top\}$ to the preorder $\mathsf{2} = \{\bot < \top\}$ has $\eps_{b,b}$ an isomorphism for all $b \in \mathsf{2}$, but is not Cauchy dense. Nevertheless, there is a similar result for general $\cat{V}$ if one assumes that $F$ is split-full -- this is the content of Subsection~\ref{ssec:split-full}.
\end{remark}

\subsection{Between preorders}\label{ssec:preorders}
Let $\sf 2 = \{\bot < \top\}$ denote the walking arrow category, which is cartesian closed. A $\sf 2$-category is a preorder, and a $\sf 2$-functor $f \colon A \to B$ is an order preserving map.

\begin{proposition}
    A $\sf 2$-functor $f \colon A \to B$ is Cauchy dense iff it is essentially surjective.
\end{proposition}
\begin{proof}
    Using Proposition~\ref{prop:V_preorder}, the condition that $f$ be Cauchy dense translates in this context to 
    \[\bigvee_a [(fa \leq b) \wedge (b \leq fa)] = \top \qquad \forall b \in B.\]
    This is equivalent to there being some $a \in A$ such that $fa \cong b$, in the sense that $fa \leq b$ and $b \leq fa$.
\end{proof}

Of course, one can also view preorders as ordinary categories, instead of $\sf 2$-categories. This is thanks to the strong monoidal functor $\mathsf{2}(\top,-) \colon \mathsf{2} \to \Set$, which gives an inclusion $\mathsf{2}\hy\Cat \to \Cat$. For ordinary functors between preorders in $\Cat$, Cauchy density is a stronger condition than for $\mathsf{2}$-functors.

\begin{proposition}
    An ordinary functor $f \colon A \to B$ between preorders is Cauchy dense iff
    \[\pi_0 \{ a \in A \mid b \leq fa \leq b' \} = \begin{cases}
        1 & \text{if } b \leq b', \\
        \varnothing & \text{otherwise.}
    \end{cases}\]
\end{proposition}
\begin{proof}
    Using the computation of coends in $\Set$, we see that $f$ is Cauchy dense iff the obvious diagram
    \[\begin{tikzcd}
    \displaystyle\bigsqcup_{a \leq a'} B(fa',b') \times B(b,fx)
    \ar[r, shift left] \ar[r, shift right] &
    \displaystyle\bigsqcup_a B(fa, b') \times B(b,fa) 
    \ar[r] & 
    B(b,b'),
    \end{tikzcd}\]
    is a coequaliser diagram for all $b,b' \in B$. If $b \not\leq b'$, then the diagram is constant at $\varnothing$, so it is always a coequaliser diagram. Otherwise, note that the category of elements of the parallel arrows has an object $\underline{a}$ whenever $b \leq fa \leq b'$, an object $\underline{aa'}$ whenever $a \leq a'$, $b \leq fa$ and $fa' \leq b'$, and nonidentity morphisms $\underline{aa'} \to \underline{a}$ and $\underline{aa'} \to \underline{a'}$ for each object of the second type. The coequaliser of the parallel arrows is a singleton precisely when this category is connected, which happens precisely when $\{a \in A \mid b \leq fa \leq b'\}$ is connected as a full subcategory of $A$.
\end{proof}

It follows that Cauchy density is not preserved under base change in general.

\begin{example}\label{ex:discrete_CD}
    The Cauchy dense functors between discrete (ordinary) categories are the bijective-on-objects ones.
\end{example}

\subsection{Between monoids}

Let $\cat{V} = \Set$ and $f \colon A \to B$ be a monoid homomorphism, seen as a functor between one-object categories. Then $f$ is Cauchy dense iff the multiplication of $B$ induces a bijection
\[\begin{tikzcd}
\displaystyle\frac{B \times B}{\ang{(bf(a),b') \sim_f (b, f(a)b') \mid a \in A}} \ar[r, "\alpha"] & B,
\end{tikzcd}\]
where this is the quotient of $B \times B$ by the equivalence relation $\approx_f$ generated by the relation $\sim_f$ with $(bf(a), b') \sim_f (b,f(a)b')$ for all $a \in A$. This function is always surjective, because composition $B \times B \to B$ is surjective. Hence, $f$ is Cauchy dense iff $\alpha$ is injective. In fact, we can simplify this condition even further.

Note that $\sim_f$ is already reflexive (take $a = 1$) and transitive (because we can compose in $A$). Moreover, if $(b_0,b_0') \sim_f (b_1,b_1')$ then $(cb_0,b_0'd) \sim_f (cb_1,b_1'd)$, with all variables in $B$, which easily implies that the same is true with $\sim_f$ replaced by $\approx_f$. 

\begin{proposition}\label{prop:monoid_CD_iff}
    Let $f \colon A \to B$ be a monoid homomorphism. Then $f$ is Cauchy dense iff $(b,1) \approx_f (1,b)$ for all $b \in B$, with the notation above.
\end{proposition}
\begin{proof}
    Since $\alpha[b,1] = b = \alpha[1,b]$, the forward implication is clear. Now, assuming $(b,1) \approx_f (1,b)$ for all $b \in B$, we show that $\alpha$ is injective. Let $b_0,b_0',b_1,b_1' \in B$ be such that $\alpha[b_0,b_0'] = \alpha[b_1,b_1']$, i.e.\ $b_0b_0' = b_1b_1'$. By the previous paragraph, $(b_0',1) \approx_f (1,b_0')$ implies $(b_0b_0',1) \approx_f (b_0,b_0')$ and, similarly, we also have $(b_1b_1',1) \approx_f (b_1,b_1')$. Then
    \[(b_0,b_0') \approx_f (b_0b_0',1) = (b_1b_1',1) \approx_f (b_1,b_1')\]
    as needed.
\end{proof}

If $A$ is a group, then $\sim_f$ is already symmetric, since
\[(b,f(a)b') = (b f(a)f(a^{-1}), f(a)b') \sim_f (bf(a),f(a^{-1})f(a) b') = (bf(a),b')\]
for all $b, b' \in B$ and $a \in A$. In this case, $\sim_f$ and $\approx_f$ coincide. Then $f$ is Cauchy dense iff $(b,1) \sim_f (1,b)$ for all $b \in B$. If $(b,1) \sim_f (1,b)$, then there is some $a \in A$ such that $f(a)1 = b$. Hence:

\begin{corollary}\label{cor:group_CD}
    Let $f \colon A \to B$ be a monoid homomorphism with $A$ a group. Then $f$ is Cauchy dense iff it is surjective. If this is the case, then $B$ is also a group.
\end{corollary}

For general monoids, we do not have such a simple characterisation. The inclusion of additive monoids $\N \to \Z$ is Cauchy dense, as is any homomorphism $f \colon A \to B$ where $B$ is generated under multiplication by elements of the form $f(a)$ and $f(a)^{-1}$ with $a \in A$. We do however have:

\begin{proposition}\label{prop:monoid_CD_implies_epi}
    Let $f \colon A \to B$ be a monoid homomorphism. If $f$ is Cauchy dense, then $f$ is an epimorphism in $\Mon$.
\end{proposition}
\begin{proof}
    We will in fact prove something stronger: for any two functions $g,h \colon \begin{tikzcd}[cramped, column sep=small] B \ar[r, shift left] \ar[r, shift right] & C \end{tikzcd}$ where $C$ is a monoid and $gf = hf$ is a monoid homomorphism, we have $g = h$. The function $g \times h \colon B \times B \to C \times C$ descends to a function $(B \times B)/{\approx_f} \to (C \times C)/{\approx_{gf}}$, because
    \[(g(b)gf(a),h(b')) \sim_{gf} (g(b), gf(a)h(b')) = (g(b), hf(a)h(b'))\]
    for all $b, b' \in B$ and $a \in A$. If $f$ is Cauchy dense, then $(b,1) \approx_f (1,b')$, and then $(g(b),1) \approx_{gf} (1,h(b))$ for all $b \in B$. The function $(C \times C)/{\approx_{gf}} \to C$ induced by the multiplication of $C$ then gives $g(b) = h(b)$, as needed.
\end{proof}

\begin{remark}
    The previous proposition does not immediately follow from the fact that Cauchy dense ordinary functors are lax epimorphisms in $\Cat$. That only ensures that if $g,h \colon \begin{tikzcd}[cramped, column sep=small] B \ar[r, shift left] \ar[r, shift right] & C \end{tikzcd} \in \Mon$ are such that $gf = hf$, then there is an invertible $u \in C$ such that $ug(b) = h(b)u$ for all $b \in B$.
\end{remark}

\noindent
\textbf{Open problem.} Is every epimorphism in $\Mon$ Cauchy dense?

\vspace{7pt}
This seems unlikely to the author, but a counterexample has not been found. It would for example imply that every monoid epimorphism has the stronger property in the proof of Proposition~\ref{prop:monoid_CD_implies_epi}.

\subsection{Connected components}

This subsection focuses once again on the $\cat{V} = \Set$ case. In particular, we will show that a Cauchy dense ordinary functor induces a bijection between the sets of connected components. In this subsection, we will write $\pi_0 \colon \Cat \to \Set$ for the connected components functor.

\begin{theorem}\label{thm:pi_0}
    Let $F \colon \cat{A} \to \cat{B}$ be a Cauchy dense ordinary functor. Then $\pi_0 F$ is bijective.
\end{theorem}
\begin{proof}
    First, assume $F$ is bijective-on-objects, and write $F^{-1}$ for the inverse of its object function. Let $b, b' \in \cat{B}$. The coend $(F_* \otimes F^*)(b,b')$ can be written as a disjoint union of coends, one for each connected component of $\cat{A}$:
    \begin{equation}\label{eq:disjoint}
        (F_* \otimes F^*)(b,b') = \bigsqcup_{C \in \pi_0 \cat{A}} \int^{a \in C} \cat{B}(Fa,b') \times \cat{B}(b,Fa).
    \end{equation}
    (A quick way to justify this is that the functor that sends a category $\cat{C}$ to the opposite of its subdivision category, in the language of Mac Lane~\cite[p.~224]{MacLane1998}, preserves coproducts.)

    We will show that $a$ and $a'$ are connected in $\cat{A}$ iff $Fa$ and $Fa'$ are connected in $\cat{B}$. The forward implication is true for any functor. Now suppose there is a morphism $f \colon Fa \to Fa'$ in $\cat{B}$. We have two equivalence classes $[f,1_{Fa}]$ and $[1_{Fa'},f]$ in $(F_* \otimes F^*)(Fa,Fa')$ which are mapped to $f$ by $\eps_{Fa,Fa'}$. Since $F$ is Cauchy dense, they must actually be the same equivalence class. In particular, $[f,1_{Fa}]$ and $[1_{Fa'},f]$ must lie in the same component of the disjoint union in~\eqref{eq:disjoint}, i.e.\ $a$ and $a'$ are connected. If instead of a morphism $Fa \to Fa'$ there is a zig-zag of morphisms, one can apply the same argument to each morphism in the zig-zag to conclude that $a$ and $a'$ are connected.

    Since $F$ is bijective-on-objects, it follows that $\pi_0F$ is bijective.
    
    Now we relax the assumption that $F$ is bijective-on-objects. Write $\cat{A} \xrightarrow{G} \im F \xrightarrow{H} \cat{B}$ for the factorisation of $F$ into a bijective-on-objects functor followed by a fully faithful one. By Proposition~\ref{prop:closure}\ref{part:GF_cd_G_ff}, since $H$ is fully faithful, $G$ is Cauchy dense, and then $H$ is Cauchy dense by part~\ref{part:GF_F_cd} of the same proposition. We have already shown that $\pi_0 G$ is bijective. By Theorem~\ref{thm:ffCD}, $\cat{B}$ is a full subcategory of the idempotent completion $\overline{\im F}$ of $\im F$ containing $\im F$ itself. In other words, the objects of $\cat{B}$ include all of $\im F$ and the splitting of some (but possibly not all) idempotents in $\im F$. Splitting certain idempotents does not change the connected components of $\im F$, so we conclude that $\pi_0 H$ is bijective.
\end{proof}

\begin{corollary}\label{cor:disj_union}
    An ordinary Cauchy dense functor $F \colon \cat{A} \to \cat{B}$ can be written as $\sqcup_i F_i \colon \sqcup_i \cat{A}_i \to \sqcup_i \cat{B}_i$ where each $F_i$ is Cauchy dense and each $\cat{A}_i$ and $\cat{B}_i$ are connected. Conversely, a functor of the form $\sqcup_i G_i$ is Cauchy dense iff each $G_i$ is Cauchy dense.
\end{corollary}
\begin{proof}
    Since $\pi_0 F$ is a bijection, $F = \sqcup_{i \in \pi_0 \cat{A}} F_i$ where $F_i$ is the restriction of $F$ to $i \in \pi_0 \cat{A}$ factorised through $\pi_0 F(i) \in \pi_0 \cat{B}$. It is easy to see that each $F_i$ is Cauchy dense if $F$ is from~\eqref{eq:disjoint}. Similarly, it is clear that a disjoint union of Cauchy dense functors is Cauchy dense.
\end{proof}

\begin{proposition}
    A functor whose domain is a groupoid is Cauchy dense iff it is equivalent to a disjoint union of surjective group homomorphisms.
\end{proposition}
\begin{proof}
    Corollary~\ref{cor:disj_union} reduces our task to showing that a Cauchy dense functor whose domain is a group must be equivalent to a surjective group homomorphism. 
    
    Let $F \colon A \to \cat{B}$ be Cauchy dense, with $A$ a group, and let $A \xrightarrow{G} \im F \xrightarrow{H} \cat{B}$ be its factorisation into a bijective-on-objects functor followed by a fully faithful one. As before, both $G$ and $H$ are Cauchy dense because $F$ is. By Corollary~\ref{cor:group_CD}, $G$ is a surjective group homomorphism. By Theorem~\ref{thm:ffCD}, $\cat{B}$ is equivalent to a full subcategory of the idempotent completion of $\im F$ containing $\im F$. Since $\im F$ is a group, it is already idempotent complete, so $H$ must be an equivalence. Thus, $F$ is equivalent to a surjective group homomorphism.
\end{proof}

In particular:

\begin{corollary}
    A Cauchy dense functor whose domain is discrete is an equivalence.
\end{corollary}

\subsection{Split-full Cauchy dense functors}\label{ssec:split-full}

We conclude by showing that for split-full functors, much like when $\cat{V}_0$ is a preorder, Cauchy density is equivalent to an apparently much weaker condition. This simplifies checking the hypotheses of Lemma~\ref{lem:sfCD_Cauchy_completion}. The proof, which is straightforward when $\cat{V} = \Set$, is slightly more involved in the general case. Note once again that in the simpler condition we quantify over a single $b \in \cat{B}$, instead of a pair of them.

\begin{theorem}\label{thm:split_full_CD}
Let $\cat{A}$ be a small category, and $F \colon \cat{A} \to \cat{B}$ be a split-full functor. Then $F$ is Cauchy dense iff 
\[\eps_{b,b} \colon \int^a \cat{B}(Fa,b) \otimes \cat{B}(b,Fa) \to \cat{B}(b,b)\]
is a split epimorphism for all $b \in \cat{B}$.
\end{theorem}
\begin{proof}
    The forward implication is obvious, so we prove converse. 
    
    For $b_0,\dots,b_n \in \cat{B}$, we will use the abbreviation
    \[\cat{B}(b_0,\dots,b_n) \coloneqq \cat{B}(b_{n-1},b_n) \otimes \cdots \otimes \cat{B}(b_0,b_1).\]
    We write $M_{b_1} \colon \cat{B}(b_0,b_1,b_2) \to \cat{B}(b_0,b_2)$ for the composition map, and similarly for longer tuples. 
    
    Let $b,b' \in \cat{B}$. First, we will show that $\eps_{b,b'}$ is a split epimorphism. The square
    \[\begin{tikzcd}
    \int^a \cat{B}(b,Fa,b,b') 
    \ar[r, "1 \otimes \eps_{b,b}"] 
    \ar[d, "\int^a M_b"']
    &[1em] \cat{B}(b,b,b')
    \ar[d, "M_b"] \\
    \int^a \cat{B}(b,Fa,b') 
    \ar[r, "\eps_{b,b'}"]
    & \cat{B}(b,b')
    \end{tikzcd}\]
    commutes by associativity in $\cat{B}$. Now, $1 \otimes \eps_{b,b}$ is a split epimorphism by assumption, and so is $M_b$ by the right unit axiom in $\cat{B}$. It follows that $\eps_{b,b'}$ is also a split epimorphism.

    It now suffices to show that $\eps_{b,b'}$ is monic. The diagram
    \[\begin{tikzcd}
    \cat{B}(Fa,b') \otimes \cat{A}(a',a) \otimes \cat{B}(b,Fa') \ar[dr, "1 \otimes F_{a',a} \otimes 1", near end] \\
    &[-8em] \cat{B}(b,Fa',Fa,b') 
    \ar[r, "M_{Fa}"] \ar[d, "M_{Fa'}"'] &[1em]
    \cat{B}(b,Fa',b')
    \ar[d, "c_{a'}"]\\
    & \cat{B}(b,Fa,b')
    \ar[r, "c_a"'] &
    \int^a \cat{B}(b,Fa,b')
    \end{tikzcd}\]
    commutes when $c$ is a cowedge, and in particular when $c$ is the coend cowedge. Since $1 \otimes F_{a',a} \otimes 1$ has a right inverse by assumption, the square alone also commutes. Seeing each of the composites in the square as an appropriate cowedge shows that the two maps
    \[\begin{tikzcd}[row sep=-0.5em]
    &[2em] \int^a \cat{B}(b,Fa,b')
    \ar[dd, equals] \\
    \int^{a,a'} \cat{B}(b,Fa',Fa,b')
    \ar[ur, shift left=1.5, "\int^a 1 \otimes \eps_{b,Fa}", start anchor=east, near end, end anchor=south west]
    \ar[dr, shift right=1.5, "\int^{a'} \eps_{Fa',b'} \otimes 1"', start anchor=east, end anchor=north west, near end] \\
    & \int^{a'} \cat{B}(b,Fa',b')
    \end{tikzcd}\]
    coincide.

    The naturality of $\eps$ ensures that the two squares in the diagram
    \[\begin{tikzcd}[column sep=-1em]
    & \int^{a,a'} \cat{B}(b,Fa',b,Fa,b')
    \ar[dl, "\int^a 1 \otimes \eps_{b,b}"', near end] 
    \ar[dd, "\int^{a,a'} M_b"]
    \ar[dr, "\int^{a'} \eps_{b,b'} \otimes 1", very near end] \\
    \int^a \cat{B}(b,b,Fa,b')
    \ar[dd, shift right=6, "\int^a M_b"']
    && \int^{a'} \cat{B}(b,Fa',b,b')
    \ar[dd, shift left=6, "\int^{a'} M_b"] \\[2pt]
    & \int^{a,a'} \cat{B}(b,Fa',Fa,b')
    \ar[dl, "\int^a 1 \otimes \eps_{b,Fa}"', pos=0.7, outer sep=-2pt] 
    \ar[dr, "\int^{a'} \eps_{Fa',b'} \otimes 1", pos=0.8, outer sep=-1pt] \\
    \int^a \cat{B}(b,Fa,b')
    \ar[rr, equals]
    && \int^{a'} \cat{B}(b,Fa',b')
    \end{tikzcd}\]
    commute, so that the whole diagram commutes. Distributing tensors over coends, the top leftmost map may in fact be written as $1 \otimes \eps_{b,b}$, and the top rightmost map is just $\eps_{b,b'} \otimes 1$. Both maps on the left-hand side have right inverses. Writing $\eps_{b,b}^r$ for a right inverse to $\eps_{b,b}$, the composite
    \[\begin{tikzcd}[column sep=4em]
    \int^a \cat{B}(b,Fa,b') 
    \ar[r, "r^{-1}"]
    & \int^a \cat{B}(b,Fa,b') \otimes I 
    \ar[r, "1 \otimes j_b"]
    & \int^a \cat{B}(b,b,Fa,b') 
    \ar[d, "1 \otimes \eps_{b,b}^r"] \\
    \int^{a'} \cat{B}(b,Fa',b')
    & \int^{a'} \cat{B}(b,Fa',b,b') \ar[l, "\int^{a'} M_b"']
    & \int^{a,a'} \cat{B}(b,Fa',b,Fa,b')
    \ar[l, "\eps_{b,b'} \otimes 1"']
    \end{tikzcd}\]
    is the identity, where $r$ is the right unitor of $\cat{V}$. Then, for any morphism $x \colon X \to \int^a \cat{B}(b,Fa,b')$, we have
    \begin{align*}
    x &= \left(\textstyle\int^{a'} M_b\right) \circ (\eps_{b,b'}\otimes \eps_{b,b}^r j_b) \circ r^{-1} \circ x \\
    &= \left(\textstyle\int^{a'} M_b\right) \circ (\eps_{b,b'}\otimes \eps_{b,b}^r j_b) \circ  (x \otimes 1) \circ r^{-1} \\
    &= \left(\textstyle\int^{a'} M_b\right) \circ (\eps_{b,b'}x\otimes \eps_{b,b}^r j_b) \circ r^{-1}.
    \end{align*}
    In particular, $x$ is determined by $\eps_{b,b'}x$, so $\eps_{b,b'}$ is monic.
\end{proof}

Of course, a split-full Cauchy dense functor may not be fully faithful. For example, a surjective group homomorphism is split-full and Cauchy dense (by Corollary~\ref{cor:group_CD} above), but only fully faithful if it is an isomorphism.

\bibliography{C:/Users/drnai/OneDrive/Documents/References/monads}

\end{document}